\documentclass[microtype]{gtpart}
\usepackage[top=1in, left=1.5in, right=1.25in, bottom=2in]{geometry}

\usepackage{hyperref}  
\hypersetup{%
  bookmarksnumbered=true,%
  bookmarks=true,%
  colorlinks=black,%
  linkcolor=black,%
  citecolor=black,%
  filecolor=black,%
  menucolor=black,%
  pagecolor=black,%
  urlcolor=black,%
  pdfnewwindow=true,%
  pdfstartview=FitBH}

\usepackage{amssymb,epsfig,amsmath,latexsym,amsthm}
\usepackage[percent]{overpic}

\theoremstyle{plain}
\newtheorem{theorem}{Theorem}[section]
\newtheorem{lemma}[theorem]{Lemma}

\newtheorem{proposition}[theorem]{Proposition}
\newtheorem{corollary}[theorem]{Corollary}

\theoremstyle{definition}
\newtheorem{definition}[theorem]{Definition}

\DeclareMathOperator{\radius}{radius}

\newcommand{\BH}{\mathbb H}
\newcommand{\BR}{\mathbb R}

\newcommand{\partialinfty}{\partial_\infty}

\begin{document}

\title{The Two Eyes Lemma: a linking problem for horoball necklaces}

\thanks{Version 1.0, May 4, 2018}

\author[Gabai]{David Gabai}
\address{Department of Mathematics\\ Princeton University\\Princeton, NJ 08544}
\email{gabai@math.princeton.edu}
\urladdr{https://www.math.princeton.edu/directory/david-gabai}

\author[Meyerhoff]{Robert Meyerhoff}
\address{Math Department\\ Maloney Hall, Fifth Floor\\ 140 Commonwealth Avenue\\ Chestnut Hill, MA 02467}
\email{robert.meyerhoff@bc.edu}
\urladdr{https://www2.bc.edu/robert-meyerhoff/}

\author[Yarmola]{Andrew Yarmola}
\address{Department of Mathematics\\ Princeton University\\ Princeton, NJ 08544}
\email{yarmola@princeton.edu}
\urladdr{https://www.uni.lu/~yarmola}

\begin{abstract} In the course of our work on low-volume hyperbolic 3-manifolds, we came upon a linking problem for horoball necklaces in $\mathbb{H}^3$. A horoball necklace is a collection of sequentially tangent beards (i.e. spheres) with disjoint interiors lying on a flat table (i.e. a plane) such that each bead is of diameter at most one and is tangent to the table. In this note, we analyze the possible configurations of an 8-bead necklace linking around two other diameter-one spheres on the table. We show that all the beads are forced to have diameter one, the two linked spheres are tangent, and that each bead must kiss (i.e. be tangent to) at least one of the two linked spheres. In fact, there is a 1-parameter family of distinct configurations.
\end{abstract}

\maketitle


\section{Introduction}

Start with a disc $D$ of radius $r$ in the Euclidean plane.  What is the maximal number of discs of radius $r$ with disjoint interiors that each {\it kiss} $D$? We say two discs {\it kiss} if they intersect on their boundaries but not in their interiors.  The answer is 6, as can be seen by noting that the visual angle (as measured from the center of $D$) of a kissing disc is 60 degrees. Further, all such configurations are the same up to rotation about $D$, and the centers of the 6 discs are the vertices of a regular hexagon.

This leads to the classical kissing problem: what is the maximal number of equal radius spheres that simultaneously kiss a base sphere of the same radius? This question was the subject of a correspondence between Isaac Newton and James Gregory in the 17th century.  Newton thought the answer was 12 but Gregory wondered whether 13 might work.  Newton was correct, as was first proven in the nineteenth century.  One could also ask about how many essentially distinct 12-kissings there are.  It turns out that there are infinitely many that are fundamentally different and then one could ask for a description of this parameter space.  Similarly, this question is of interest in higher dimensions.  Good references for this material are the classic text ``Sphere Packings, Lattices and Groups'' by Conway and Sloane (Chapter 2) \cite{Conway:1999wc} and the semi-expository paper ``The Twelve Spheres Problem'' by Kusner, Kusner, Lagarias, and Shlosman \cite{Kusner:2016ui}.

In the course of our work on low-volume hyperbolic 3-manifolds \cite{GHMTY}, we faced a different generalization of the kissing problem.  Here we came upon a cycle (or necklace) of $\le 8$ kissing spheres (or  beads) of diameter at most one lying on a flat table.  Suppose they link around spheres $D_1$ and $D_2$, also on the table, with disjoint interiors and of height (i.e. diameter) exactly one.  As a consequence of the \emph{Two-Eyes Lemma} (see below), we are able to prove that $D_1$ and $D_2$ must kiss, that each bead must kiss $D_1$ or $D_2$, and that each bead must be of height one.  An example of this is obtained by taking a hexagonal packing of height-one spheres, labeling two abutters as $D_1$ and $D_2$, and then observing the cycle of 8 spheres encircling them.  In fact, there is a 1-parameter family of essentially different solutions that is gotten by sliding one sphere along $D_1$ (or $D_2$) and then all other sphere positions are forced.  Further, these are the only possible solutions. See Figures \ref{fig:hex} and \ref{fig:shift}. We note that when all the beads are assumed to be of height one, our result reduces to a planar problem that is quite easy to address. 

\begin{figure}[t]
  \centering
  \begin{minipage}[b]{0.45\textwidth}
    \includegraphics[width=\textwidth]{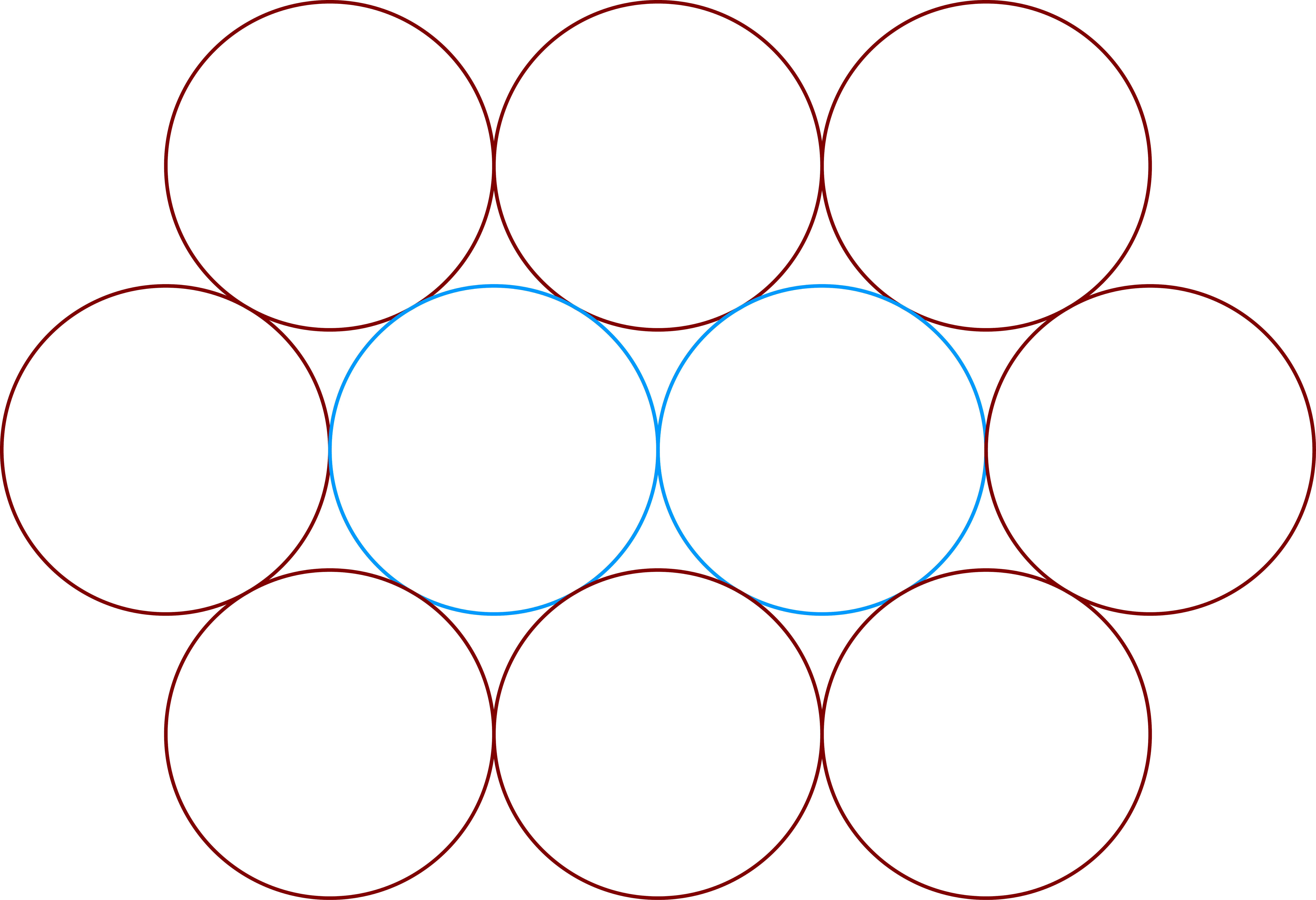}
    \caption{Hexagonal configuration.}
    \label{fig:hex}
  \end{minipage}
  \hspace*{4ex}
  \begin{minipage}[b]{0.45\textwidth}
    \includegraphics[width=\textwidth]{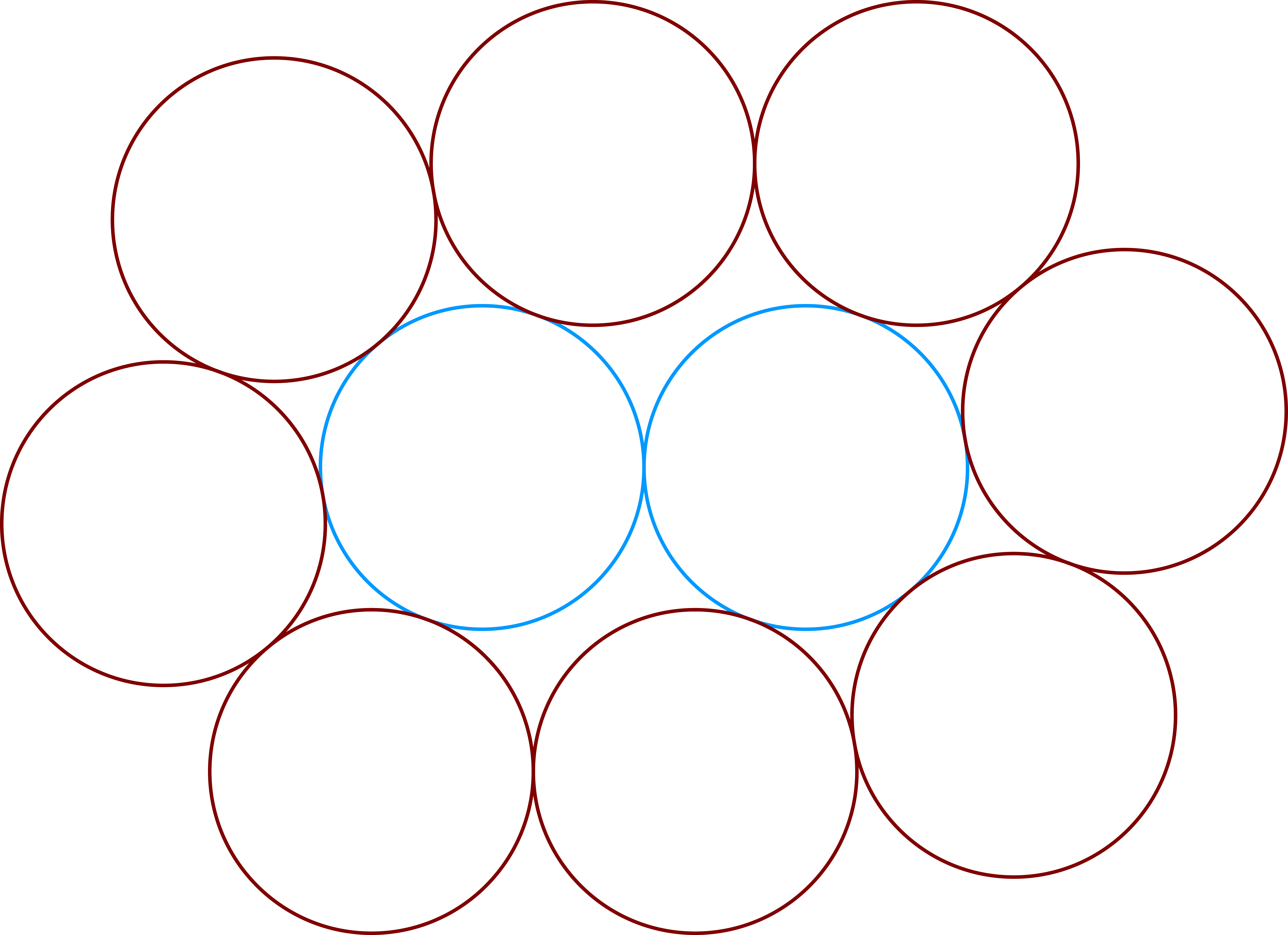}
    \caption{Non-hexagonal configuration.}
    \label{fig:shift}
  \end{minipage}
\end{figure}

We are naturally lead to the following question, which we simply pose, but do not address.  Given two abutting spheres of radius $r$ in $\mathbb{R}^3$, what is the kissing number for these two spheres? That is, what is the maximal number of (non-overlapping) radius $r$ spheres that each kiss either of the two abutting spheres?



\section{Set-Up and Statement of Main Proposition}

\begin{definition} Consider the upper-half-space model of hyperbolic 3-space $\BH^3$ with the standard projection $\pi:\BH^3\to \BR^2$.  We say that a horoball $B$ is \emph{full-sized} if $\radius(\pi(B))=1/2$ and \emph{less than full-sized} if $\radius(\pi(B))<1/2$.  Denote by \emph{center(B)} the point at infinity of $B$.\end{definition}  

\begin{definition}  A \emph{k-necklace} $\eta = N_1\cup\cdots\cup N_k$ is a cyclicly ordered set of $ k$ horoballs with disjoint interiors such that one is tangent to the next.  In what follows indices for a $k$-necklace are always modulo $k$. The $N_i$'s are called the \emph{beads} and $k$ is called the \emph{necklace} or \emph{bead number} of $\eta$. The hyperbolic geodesics connecting the centers of successive horoballs are called ${\it ties}$.
\end{definition}

In this note, we will fix a horoball $H_\infty \subset \BH^3$ centered at $\infty$ with $\partial H_\infty$ a plane of Euclidean height $1$ in the upper half-space model. We see that a horoball is full-sized (or {\it full}) if it is tangent to $H_\infty$. We would like to understand how necklaces can wind around full-sized horoballs. The main result fo this note is


\begin{proposition}\label{prop:main} If $C_1$ and $C_2$ are full-sized horoballs with disjoint interiors then the minimum bead number of a necklace $\eta$ with less-than-or-equal-to full-sized horoballs encircling $C_1$ and $C_2$ is 8.  If the bead number is 8, then all horoballs in $\eta$ must be full-sized; one example of this arises from the hexagonal packing of full-sized horoballs in the upper-half-space model.  Further, all examples with bead number 8 are obtained by sliding $N_1$ in $\eta$ along $C_i$ and then placing the remaining $N_i$ cyclically in turn making sure that each $N_i$ abuts $C_1$ and/or $C_2$.

\end{proposition}


\subsection*{Acknowledgements.}
The first author was partially supported by NSF grants DMS-1006553, 1607374. 
The second author was partially supported by NSF grant DMS-1308642.
The third author was partially supported as a Princeton VSRC with DMS-1006553.
\section{The Two-Eyes Lemma}

Since we will be working with projections of horoballs to the plane, we will need the following useful formula.

\begin{lemma}\label{lem:hd} (Horoball distance) Let $B_1, B_2$ be two horoballs with disjoint interiors in the upper half space model with $b_i = \partialinfty B_i \in \BR^2$ and of Euclidean height $h_i$. Then the hyperbolic distance $d_\BH(B_1,B_2)$ between $B_1, B_2$ is given by $$d(B_1,B_2) = \log\left(\frac{d_\mathbb{E} (b_1, b_2)^2}{h_1 h_2}\right).$$
\end{lemma}

\begin{proof} Consider the point $b_2' = b_1 - (b_2 - b_1)$ and let $\gamma$ be the geodesic between $b_2, b_2'$. Note that $\gamma$ has Euclidean radius $d_\mathbb{E} (b_1, b_2)$. The highest point of $\gamma$ lies directly above (or below) the highest point $h_1$ of $B_1$ in the upper-half-space model. In particular, we have the distance $d_\BH(\gamma, B_1) = \log \left( d_\mathbb{E} (b_1, b_2)/h_1 \right)$. Note that $d_\BH(\gamma, B_1) > 0$ if and only if  $\gamma \cap B_1  = \emptyset$. Rotating $180^\circ$ around $\gamma$ by an elliptic isometry, we see that $B_1$ has to map to a horoball at infinity of Euclidean height $d_\mathbb{E} (b_1, b_2)^2/h_1$. Since $B_1, B_2$ had disjoint interiors, it follows that $h_2 \leq d_\mathbb{E} (b_1, b_2)^2/h_1$ and $$d(B_1,B_2) = \log\left(\frac{d_\mathbb{E} (b_1, b_2)^2}{h_1 h_2}\right).$$
\end{proof}

A direct corollary of this computation is a statement about visual angles.

\begin{corollary} (Visual Angle) Let $C$ be a full-sized horoball and let $B$ be an at most full-sized horoball tangent to $C$, then the visual angle of $\pi(B)$ from center$(C)$ is $\leq \pi/3$ with equality if and only if $B$ is full-sized.
\end{corollary}

We now turn to the Two-Eyes Lemma, which is depicted in Figures \ref{fig:setup}, \ref{fig:left}, \ref{fig:right} and \ref{fig:middle}.  

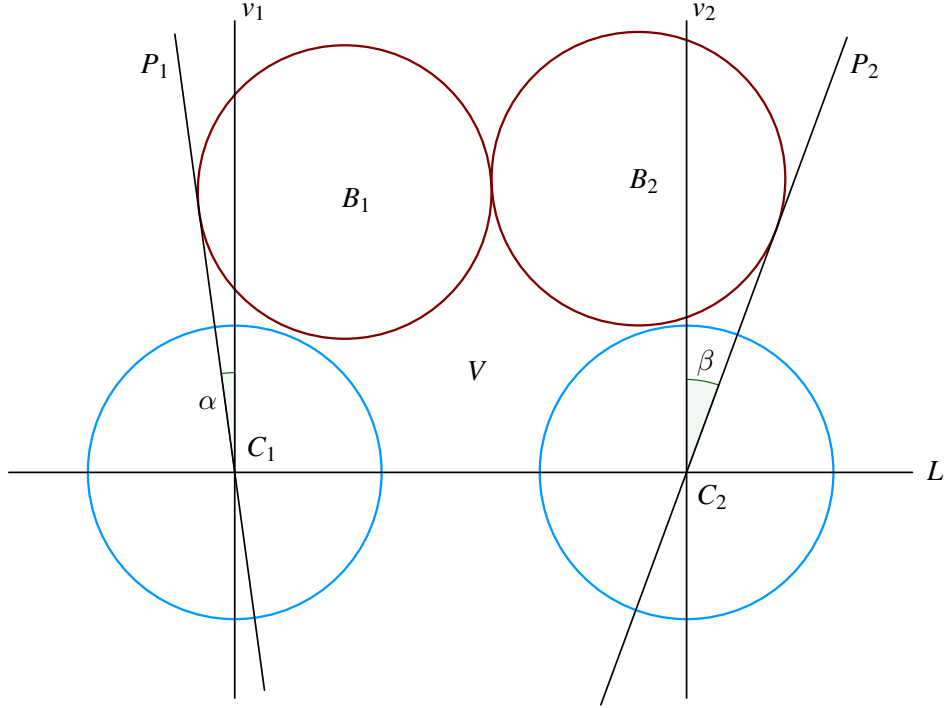
\begin{figure}[ht]
\begin{center}\begin{overpic}[scale=.15]{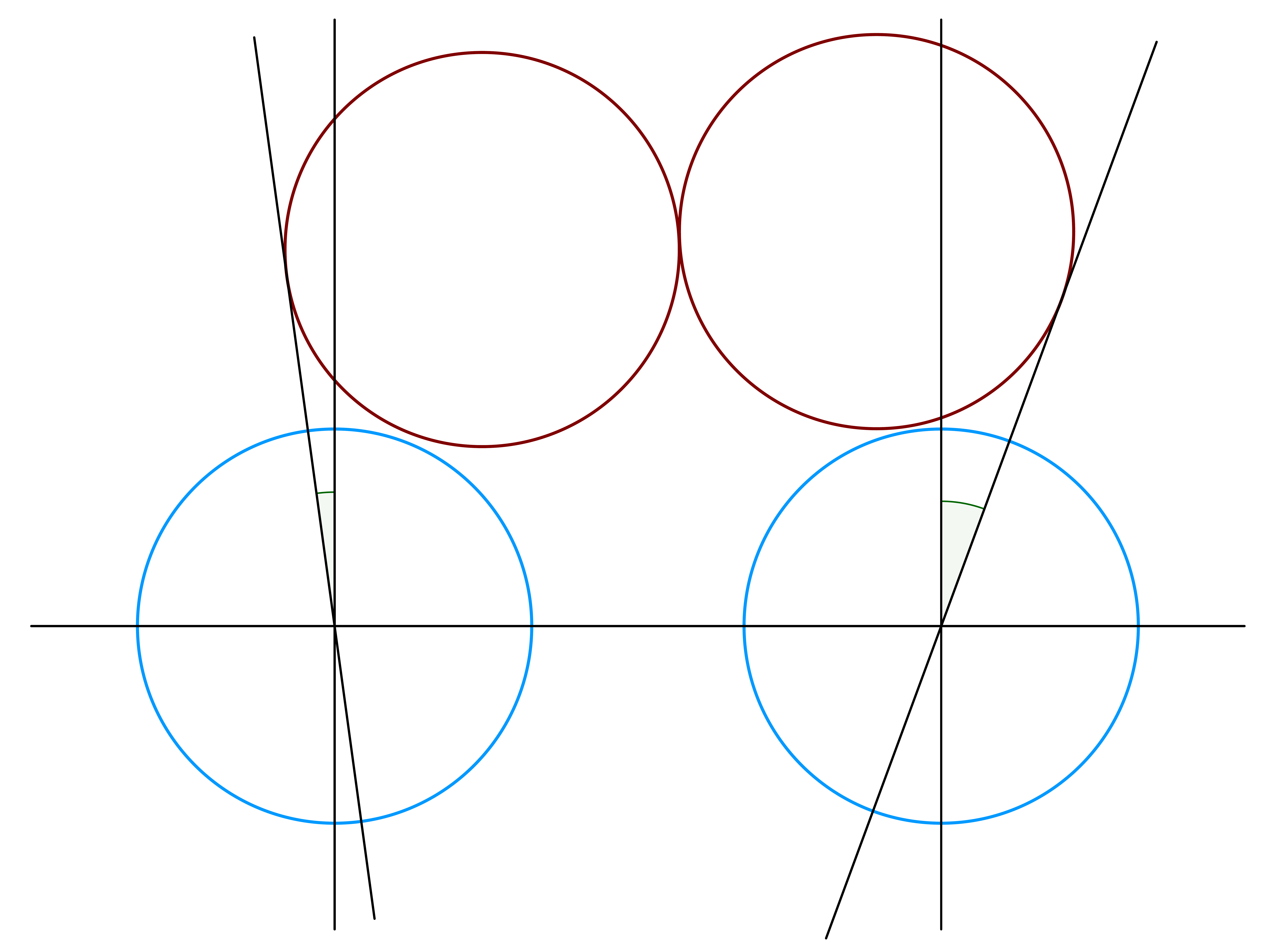}
\put(37,53){$B_1$}
\put(67,55){$B_2$}
\put(27,27){$C_1$}
\put(74,22){$C_2$}
\put(74,36){$\beta$}
\put(22,32){$\alpha$}
\put(90,67){$P_2$}
\put(16,67){$P_1$}
\put(98,24.5){$L$}
\put(50,35){$V$}
\put(26.5,73){$v_1$}
\put(73.5,73){$v_2$}
\end{overpic}
\end{center}
\caption{$\alpha + \beta \leq \pi/3$.}
\label{fig:setup}
\end{figure}

\begin{lemma}  (Two-Eyes Lemma)  Let $C_1$ and $C_2$ be full-sized horoballs with disjoint interiors. Let $B_1$ and $B_2$ be tangent horoballs with heights $h_1$ and $h_2$, respectively,  with interiors disjoint from $C_1 \cup C_2$.  Assume $h_i \leq 1$ for $i = 1,2$. Let $L$ be the line through center$(C_1) $ and center$(C_2)$ and let $v_1$ and $v_2$  be lines orthogonal to $L$ passing through center$(C_1)$ and center$(C_2)$, respectively.  Let $V_1$ and $V_2$ be the geodesic planes with boundaries containing $v_1$ and $v_2$ and let $V$ be the closure of the region bounded by  $V_1 \cup V_2$.  Suppose that for each $i, B_i\cap V_i\neq\emptyset$.  Let $P_i$ denote the line tangent to $\pi(B_i)$ through center$(C_i)$ such that $\pi(B_1\cup B_2)$ lies to one side.  Finally let $\alpha$ (resp. $\beta$) be the acute angle between $P_i$ and $v_i$.  Then,
\begin{enumerate}
\item $\alpha+\beta\le \pi/3$
\item If $\alpha + \beta = \pi/3$, then 
\begin{enumerate}
\item  $C_1$ is tangent to $C_2$
\item  $B_1$ and $B_2$ are full-sized
\item  for $i=1,2\  we have that B_i$ is tangent to $C_i $ and the line $J$ through center$(B_1)$ and center$(B_2)$ is parallel to $L$.
\end{enumerate}
\end{enumerate}
\end{lemma}

\begin{figure}[t]
  \centering
  \begin{minipage}[c][][c]{0.49\textwidth}
    \begin{overpic}[scale=1]{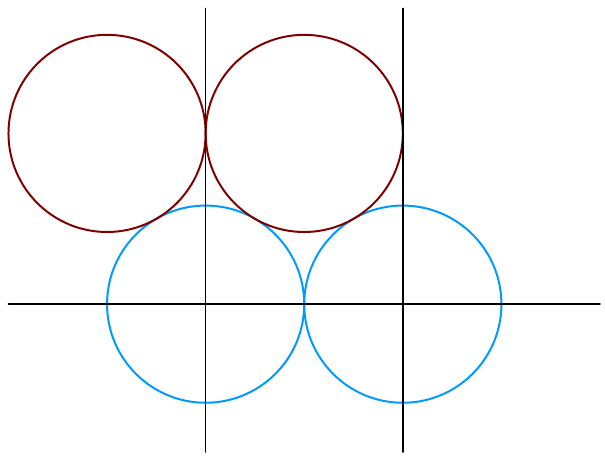}
    \put(15,52){$B_1$}
    \put(48,52){$B_2$}
    \put(36,18){$C_1$}
    \put(70,18){$C_2$}
    \put(93,27){$L$}
    \put(27,73){$v_1$}
    \put(68,73){$v_2$}
    \end{overpic}
    \caption{$\alpha = \pi/3,  \beta = 0$}
    \label{fig:left}
  \end{minipage}
  \begin{minipage}[c][][c]{0.49\textwidth}
    \begin{overpic}[scale=1]{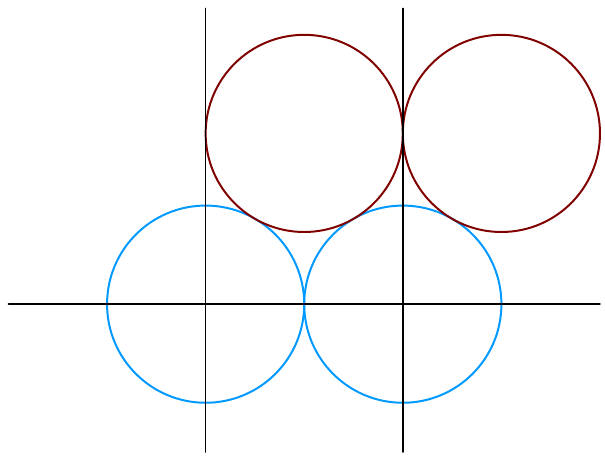} 
    \put(81,52){$B_2$}
    \put(48,52){$B_1$}
    \put(36,18){$C_1$}
    \put(70,18){$C_2$}
    \put(93,27){$L$}
    \put(27,73){$v_1$}
    \put(68,73){$v_2$}
    \end{overpic}
    \caption{$\alpha =0, \beta = \pi/3$}
    \label{fig:right}
  \end{minipage}
  \begin{minipage}[c][][c]{0.49\textwidth}
    \begin{overpic}[scale=1]{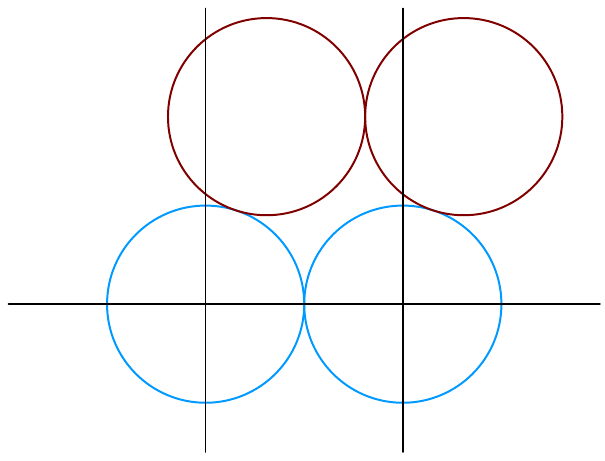} 
    \put(75,55){$B_2$}
    \put(42,55){$B_1$}
    \put(36,18){$C_1$}
    \put(70,18){$C_2$}
    \put(93,27){$L$}
    \put(27,4){$v_1$}
    \put(68,4){$v_2$}
    \end{overpic}
    \caption{$0 < \alpha, \beta < \pi/3, \alpha + \beta = \pi/3$}
    \label{fig:middle}
  \end{minipage}
\end{figure}

\begin{proof}  To start with, we can assume that $L$ is parallel to the $x$-axis.  The proof involves a series of steps whereby the positions of $B_1, B_2, C_1, C_2$ are repeatedly \emph{improved}.  The reader should note that  any \emph{improvement} strictly increases $\alpha+\beta$.  In the end $\alpha+\beta=\pi/3$ and the various horoballs satisfy the equality conclusions.  We repeatedly use the fact that an operation that moves center($B_2$) infinitesimally closer to $P_2$ is $\beta$ increasing with the analogous fact holding for $\alpha$.

Let $b = d_\mathbb{E}(\text{center}(B_1),\text{center}(B_2)), c = d_\mathbb{E}(\text{center}(C_1),\text{center}(C_2))$ and $d_{ij} = d_\mathbb{E}(\text{center}(B_i),\text{center}(C_j))$ for $i,j \in \{1,2\}$. We can assume that $\text{center}(C_1) = (0,0), \text{center}(C_2) = (c,0), \text{center}(B_1) = (x_1,y_1)$ and $\text{center}(B_2) = (x_2,y_2)$. Note that $c \geq 1$, $-h_1/2 \leq x_1 \leq h_1/2$ and $- h_2/2  \leq x_2 - c \leq h_2/2$. By Lemma \ref{lem:hd}, we also have that $b = \sqrt{h_1 h_2}$ and $d_{ij} \geq \sqrt{h_i}$, with equality if and only if $B_i$ is tangent to $C_j$.

\vskip 8pt
\noindent\emph{Step 1.}  At the cost of possibly increasing $\alpha+\beta$ we can assume that either $B_1$ is tangent to $C_1$ or $B_2$ is tangent to $C_2$.  

\vskip 8pt
\noindent\emph{Proof.}     If both $B_1\cap C_1= \emptyset$ and $B_2 \cap C_2=\emptyset$, then we can translate $B_1\cup B_2$ in the $(0,-1)$ direction until a first tangency occurs. Note that both $\alpha$ and $\beta$ increase.  If $B_1\cap C_2\neq\emptyset$ but $B_1\cap C_1 = \emptyset$, then we can obtain a contradiction as follows: we have $(x_1-c)^2 + y_1^2  = d_{12}^2 = h_1$ and $x_1^2 + y_1^2 = d_{11}^2 > h_1$. However, since $x_1 \leq h_1/2$, we obtain $1 \leq c < h_1 \leq 1$, a contradiction. A similar fact holds for $B_2$, thus the tangency is of the type claimed.\qed
 
\vskip 8pt
\noindent\emph{Step 2.}  At the cost of possibly increasing $\alpha+\beta$ we can additionally assume that either $C_1\cap C_2\neq \emptyset$ or each of $B_1$ and $B_2$ are respectively tangent to $C_1$ and $C_2$.
\vskip 8pt
\noindent\emph{Proof.} It suffices to consider the case where $B_1$ is tangent to $C_1$. If $C_2$ is disjoint from $B_2$, then translate $C_2$ in the $(-1,0)$ direction until a first tangency occurs. Note that $\beta$ increases. If $C_2$ becomes tangent to $B_1$ first, then by the computation in Step 1, $c = 1$ and $C_2$ is also tangent to $C_1$. Lastly, we observe that $B_2 \cap V_2 \neq \emptyset$ remains true as we translate by computation: if $-h_2/2  \leq x_2 - c \leq h_2/2$ fails as we decrease $c$, we have that $x_2 > c+ h_2/2$. But $x_2 \leq x_1 + b = x_1 + \sqrt{h_1h2} \leq h_1/2 + (h_1+h_2)/2$, so we obtain $1 \leq c < h_1 \leq 1$, a contradiction.\qed

\vskip 8pt
\noindent\emph{Step 3.}  At the cost of possibly increasing $\alpha+\beta$ we can further assume that for each $i, B_i\cap C_i\neq\emptyset$.

\vskip 8pt
\noindent\emph{Proof.}  It suffices to consider the case that $B_1\cap C_1\neq\emptyset$ and  $B_2\cap C_2=\emptyset$.  Let $J$ denote the ray from center$(B_1)$ through center$(B_2)$.  First assume that $J\cap P_2\neq\emptyset$.  For each $t\ge 0$ we  translate $B_2$ away from $B_1$ by moving its center Euclidean distance $t$ along $J$ away from center$(B_2)$ to obtain $B'_2(t)$.  We  then expand $B'_2(t)$ keeping its center fixed until it first hits $B_1$ to obtain $B_2(t)$.  Let $B_2$(new) be the first $B_2(t) $ that is either full-sized or satisfies $B_2(t)\cap C_2\neq\emptyset$.  Note that if $B_2$(new) $\neq B_2$, then $\beta$ increases.  We now abuse notation by denoting $B_2$(new) by $B_2$.  Thus, if $B_2\cap C_2= \emptyset$,  then $B_2$ is full-sized and by Step 2, $C_1\cap C_2\neq \emptyset$. 

If $J\cap P_2=\emptyset$, then apply a clockwise rotation about the geodesic $\gamma$ through center($B_1$) and $\infty$ until either $B_2\cap C_2\neq\emptyset$ or $J\cap P_2\neq\emptyset$.  This operation is strictly $\beta$ increasing.  If now $J\cap P_2\neq \emptyset$, then argue as in the first paragraph to conclude that either Step 3 holds or $B_2$ is full sized and $C_1\cap C_2\neq \emptyset$.  

We have now reduced to the case that $B_2$ is full-sized, $C_1\cap C_2\neq\emptyset$ and  $B_2\cap C_2=\emptyset$.  Observe that $y_2 \geq y_1$.  This is immediate if $B_1$ is full-sized.  In general, center$(B_1)$ lies on the line perpendicular to the midpoint of the segment between center$(C_1)$ and center$(B_2)$ since $B_1$ is tangent to the full-sized horoballs $C_1$ and $B_2$. Since $x_1 \leq 1/2 \leq x_2$, the maximal $y_1$ is obtained when $B_1$ is full-sized and hence $y_2 \geq y_1$.  Since $P_2$ has non-negative slope,  a clockwise rotation about $\gamma$ both transforms $B_2$ to a horoball tangent to $C_2$ and increases $\beta$.\qed

\vskip 8pt
\noindent\emph{Step 4.}     At the cost of possibly increasing $\alpha+\beta$ we can further assume that both $B_1$ and $B_2$ are full-sized.

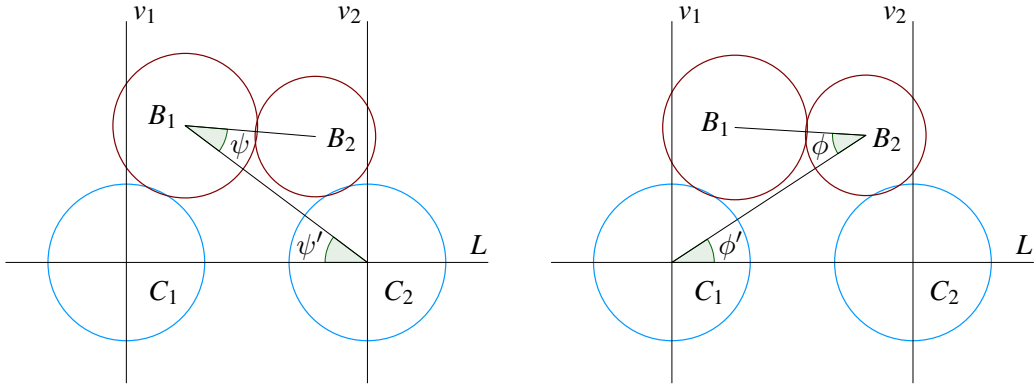
\begin{figure}[t]
  \centering
  \begin{minipage}[c][][c]{0.49\textwidth}
    \begin{overpic}[scale=0.8]{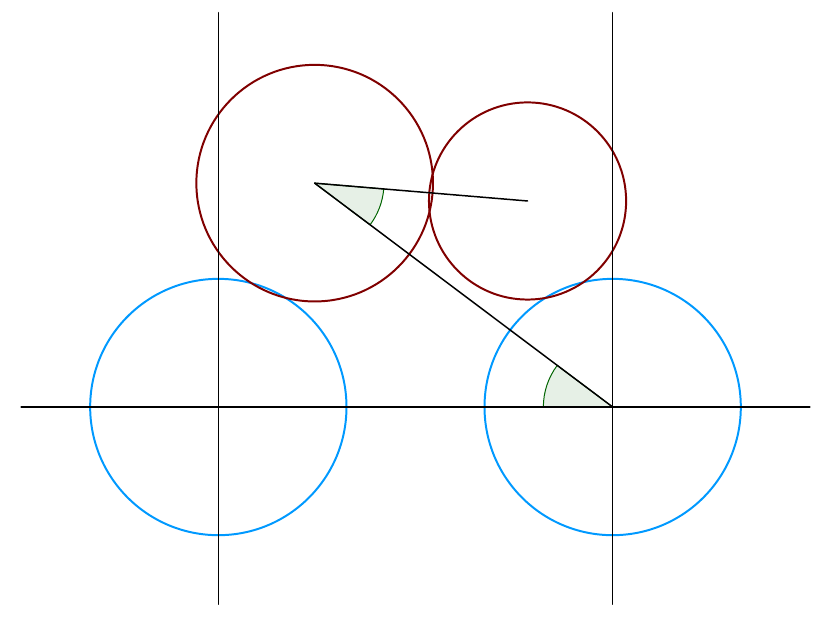}
    \put(30,52.5){$B_1$}
    \put(65,48){$B_2$}
    \put(30,18){$C_1$}
    \put(76,18){$C_2$}
    \put(93,27){$L$}
    \put(27,73){$v_1$}
    \put(67,73){$v_2$}
    \put(46,46.7){$\psi$}
    \put(59,27){$\psi'$}
    \end{overpic}
  \end{minipage} 
  \begin{minipage}[c][][c]{0.49\textwidth}
    \begin{overpic}[scale=0.8]{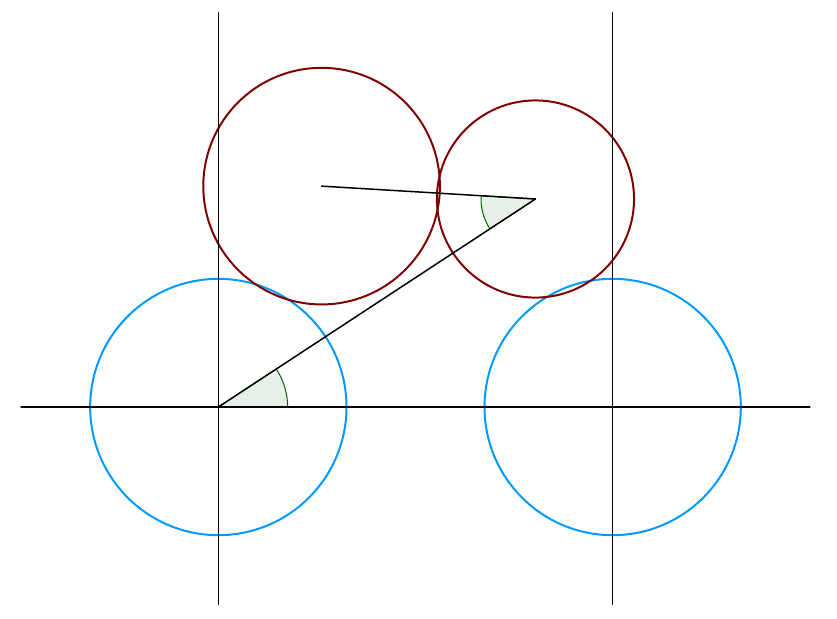} 
    \put(31.5,51){$B_1$}
    \put(65,48){$B_2$}
    \put(30,18){$C_1$}
    \put(76,18){$C_2$}
    \put(93,27){$L$}
    \put(27,73){$v_1$}
    \put(67,73){$v_2$}
    \put(53,46.5){$\phi$}
    \put(34.8,27){$\phi'$}
    \end{overpic}
  \end{minipage}
   \caption{Transforming $B_2$ by increasing $\psi$ up to $\pi/2$ increases both $\text{radius}(B_2)$ and $\beta$. Similarly for $B_1$.}
    \label{fig:angles}
\end{figure}

\vskip 8pt
\noindent\emph{Proof.}  Consider the hyperbolic geodesic $\gamma_1$ from center$(B_1)$ to center$(C_2)$ and define the angles $\phi, \phi'$ and $\psi,\psi'$ as in Figure \ref{fig:angles}.  An elliptic rotation of angle $\theta$ about $\gamma_1$ transforms $B_2$ to the horoball $B_2(\theta)$.  Being a hyperbolic isometry setwise fixing $B_1$ and $C_2$, it follows that   $B_2(\theta)$ is tangent to both $B_1$ and $C_2$.  Oriented appropriately, as $\theta$ increases so does $\psi(\theta)$, where $\psi(\theta)$ is defined as in Figure \ref{fig:angles}, where $B_2$ is replaced by $B_2(\theta)$.  If $\psi=\psi(0)<\pi/2$, the next two lemmas show that increasing $\psi$ up to $\pi/2$ strictly increases $\beta$ as well as the radius of $\pi(B_2(\theta))$.  Note that $\psi', \phi'\le \pi/3$, since say $\psi'$ is the angle at the base of a right triangle whose height is at most 1 and whose base is at least 1/2.  Since $\phi+\psi=\phi'+\psi'\le 2\pi/3$, it follows that one of $\phi$  or $\psi$ has angle at most $\pi/3$.  Without loss of generality we can assume that the latter holds.  

To prove Step 4, we first rotate $B_2$ as above so that it either becomes full-sized or $\psi=\pi/2$.  Next we rotate $B_1$ until either it becomes full-sized or $\phi=\pi/2$.   Here the rotation has axis $\gamma_2$, the geodesic from center$(B_2)$ to center$(C_1)$.  Next rotate $B_2$ so that either it becomes full sized or $\psi=\pi/2$ and so on.  After finitely many such rotations one of $B_2$ or $B_1$ becomes full sized.  Since each step involves a $\ge \pi/6$ rotation,  the process stops after a finite and computable time.  After some $B_i$ becomes full-sized, the third lemma below shows that one more rotation suffices to bring the other to full size. Thus Step 4 follows from the next three lemmas.\qed

\begin{lemma}  If $\psi(0)<\psi(\theta)\le \pi$, then $\radius(\pi(B_2(\theta)))>\radius(\pi(B_2(0)))=\radius(\pi(B_2))$.  The analogous result holds for transformations of $B_1$.\end{lemma}

\begin{proof}  Let $H_\infty$ denote the horoball $z\ge 1$.  Apply a hyperbolic isometry that takes center$(C_2)$ to $\infty$.  See Figure \ref{fig:c2infty} which shows the projections of the transformed $B_1, B_2, C_1$ and $H_\infty$ to the new $xy$ plane.   We abuse notation by continuing to call the transformed horoballs by their original names.  Notice that $H_\infty$ and $B_2$ are full-sized.  Since $\gamma_1$ is now a vertical geodesic an elliptic transformation fixing $\gamma_1$ is a Euclidean rotation in these coordinates.  A counterclockwise rotation by angle $\theta$ takes $B_2=B_2(0)$ to $B_2(\theta)$.  Consider the ideal tetrahedron $T_\theta$ with vertices center$(B_1)$, center$(B_2(\theta))$, center$(H_\infty)$, center$(C_2)$.  Since opposite dihedral angles of $T_\theta$ are equal, the angle $\psi$ in Figure \ref{fig:angles} is equal to the angle of the same name in Figure \ref{fig:c2infty}.  

In the original coordinates $\radius(\pi(B_2(\theta))) $ monotonically increases as the $\gamma_1$ dihedral angle decreases and is maximized when this angle equals $0$.   As this angle decreases to $0$ the angle $\psi(\theta)$ increases to $\pi$.  \end{proof}

\begin{figure}[t]
  \centering
  \begin{minipage}[c][][c]{0.49\textwidth}
    \begin{overpic}[scale=1]{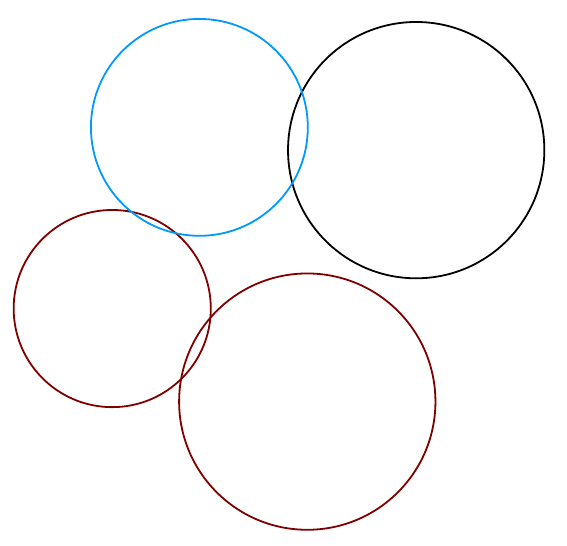}
    \put(17,39){$B_1$}
    \put(52,23){$B_2$}
    \put(32,72){$C_1$}
    \put(71,67.5){$H_\infty$}
    \end{overpic}
  \end{minipage}
  \begin{minipage}[c][][c]{0.49\textwidth}
    \begin{overpic}[scale=1]{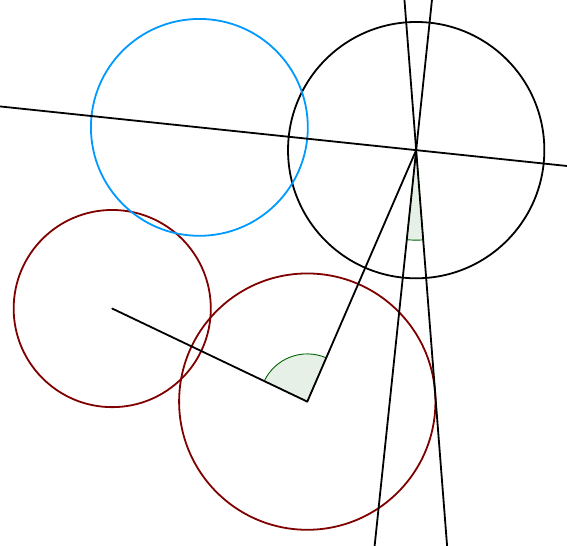} 
    \put(11,41){$B_1$}
    \put(52,18){$B_2$}
    \put(32,80){$C_1$}
    \put(80,75){$H_\infty$}
    \put(5,80){$L$}
    \put(80,5){$P_2$}
    \put(77,95){$v_2$}
    \put(48,36.5){$\psi$}
    \put(76,55){$\beta$}
    \end{overpic}
  \end{minipage}
  \caption{Sending $C_2$ to infinity and computing $\beta$.}
   \label{fig:c2infty}
\end{figure}
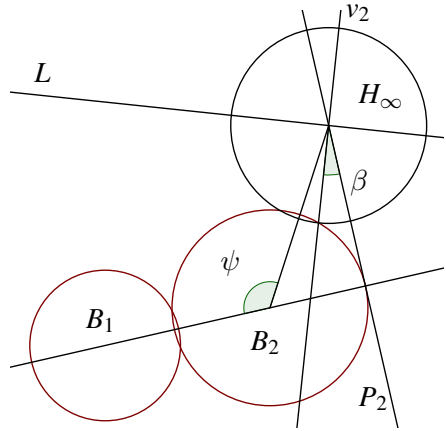
\begin{figure}[t]
  \centering
    \begin{minipage}[c][][c]{0.49\textwidth}
    \begin{overpic}[scale=1]{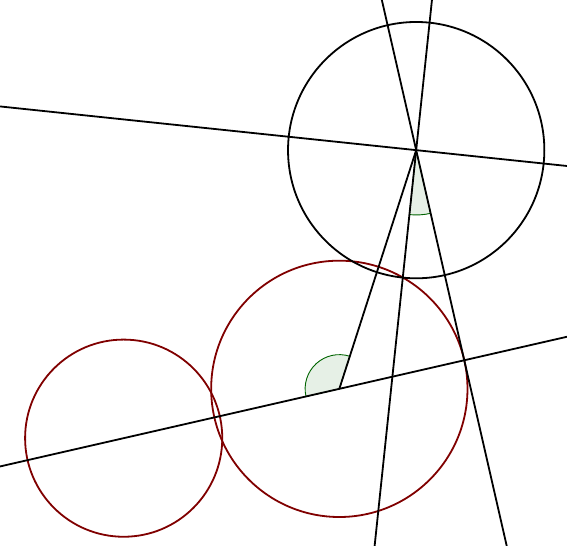} 
    \put(17,23){$B_1$}
    \put(55,18){$B_2$}
    \put(80,75){$H_\infty$}
    \put(5,80){$L$}
    \put(80,5){$P_2$}
    \put(77,95){$v_2$}
    \put(48,36.5){$\psi$}
    \put(78,55){$\beta$}
    \end{overpic}
    \caption{Maximizing $\beta$}
    \label{fig:maximized}
  \end{minipage}
\end{figure}

\begin{lemma}  If $\psi(0)<\psi(\theta)\le \pi/2$, then $\beta(\theta)>\beta(0):=\beta$.  The analogous result holds for transformations of $B_1$.\end{lemma}

\begin{proof}  Figure \ref{fig:c2infty} shows how to compute $\beta$.  Note that $\beta$ is maximized when the line from center$(B_1)$ through center$(B_2)$ is orthogonal to $P_2$ at which point $\psi>\pi/2$. See Figure \ref{fig:maximized}  \end{proof}


\begin{lemma}  If $B_1$ is full sized, then for some $\theta$ with $\psi(0)<\psi(\theta)< \pi/2, B_2(\theta)$ is full-sized.  The analogous result holds for transformations of $B_1$.\end{lemma}

\begin{proof}  Since $B_1$ is full sized it is tangent to $H_\infty$ in addition to $B_2$.  Since a horoball in the original coordinates is full sized exactly when it is tangent to $H_\infty$, we observe that $B_2(\theta)$ becomes full sized for some $\psi<\pi/2$.\end{proof}

\noindent\emph{Step 5.}  For $ i=1,2$ let $L_i$ denote the line through center$(B_i)$ and center$(C_i)$.    At the cost of possibly increasing $\alpha+\beta$ we can further assume that both $L_1$ and $L_2$ are parallel and hence $B_1$ (resp $C_1$) is tangent to $B_2$ (resp. $C_2$) and the line $J$ through the centers of $B_1$ and $B_2$ is parallel to $L$.

\vskip 8pt

\noindent\emph{Proof.}  A clockwise rotation of $\BH^3$ applied to $B_2$ using the vertical geodesic  through $C_2$ as axis takes $L_2$ to a line parallel to $L_1$.  Let $B_2'$ denote the rotated $B_2$.  This operation increases $\beta$ and makes $J$ parallel to $L$ but loses the $B_1, B_2$ tangency.  Next translate $B_2'$ and $C_2$ in the $(-1,0)$ direction until the translated $C_2$ becomes tangent to $C_1$, in which case the translated $B_2'$ also becomes tangent to $B_1$.  \qed

\vskip 8pt

\noindent\emph{Step 6.}  $\alpha+\beta=\pi/3$ and the conclusions (a)-(c) also hold.

\vskip 8pt
\noindent\emph{Proof.}  We have already shown that conclusions (a)-(c) hold.  Since $B_2$ is full-sized and tangent to $C_2$, the visual angle of $\pi(B_2)$ from center$(C_2)$ is equal to $\pi/3$.  Using the fact that $B_1\cup C_1$ is a translate of $B_2\cup C_2$ it follows that this visual angle decomposes into $\alpha+\beta.$ \qed

This completes the proof of the Two-Eyes Lemma.\end{proof}

\section{Proof of Proposition \ref{prop:main}}

The proof of the main proposition is now just a counting argument. 

\noindent\emph{Proof of Proposition \ref{prop:main}}   As in the proof of the Two-Eyes Lemma, consider the hyperplanes $V_1,$ and $V_2$. Since the necklace $\eta$ winds around $C_1$ and $C_2$, it follows that $V_1$ and $V_2$ each intersect at least two horoballs of $\eta$. For $i = 1,2$, let $B_i^U, B_i^L$ be these horoballs intersecting $V_i$ with centers in the upper and lower half-planes, respectively. These four horoballs are distinct. Further, we can assume that $B_2^U, B_2^L$ have the largest $x$-coordinates and $B_1^U, B_1^L$ are the smallest $x$-coordinates amongst all choices in $\eta$ satisfying the non-empty intersection conditions. Since all horoballs are at most full-sized, visual angle around center$(C_i)$ tells us that, away from the critical case where {\it both} $B_i^U$ and $B_i^L$ are tangent to $V_i$, we need at least two more horoballs to connect $B_1^L$ to $ B_1^U$ and at least two more to connect $B_2^U$ to $B_2^L$ in the clockwise direction along $\eta$. Away from this critical case, the necklace must have at least $8$ horoballs.

Assume we are in the critical case where $B_i^U$ and $B_i^L$ are tangent to $V_i$ for some $i$. Without loss of generality, we can take $i =1$. By the minimality of the $x$-coordinates and the fact that necklace horoballs are sequently tangent, we can assume that $B_1^U$ and $B_1^L$ lie entirely to the left of $V_1$, aside from the points of tangency. The region $V$, between $V_1$ and $V_2$, will then contain at least two more horoballs, but these cannot be $B_2^U, B_2^L$ by the maximality of the $x$-coordinate and because all the horoballs are at most full-sized. Therefore, we need at least 1 horoball to join $B_1^L$ to $B_1^U$, 2 more horoballs in $V$, and at least 1 more horoball to join $B_2^U$ to $B_2^L$, giving us a total of $8$.

We turn to the case where $\eta$ has exactly $8$ horoballs. It remains to show that all are full sized and the configuration is obtained by sliding the hexagonal example. For this, we will use the Two-Eyes Lemma and visual angle arguments. Assume that in each of the pairs $\{B_1^U, B_2^U\}$ and $\{B_2^L, B_1^L\}$ at least one of the horoballs is not tangent to the associated $V_i$. In this setting, our counting argument in the first paragraph gives that the horoballs $B_1^U$ and $B_2^U$ are tangent. Similarly for $B_1^L$ and $B_2^L$. Let $\alpha, \beta$ be the angles from the Two-Eyes Lemma applied to the pair  $\{B_1^U, B_2^U\}$ and $\alpha', \beta'$ be the angles for the pair $\{B_2^L, B_1^L\}$. 
It follows that $\alpha + \beta \leq \pi/3$ and $\alpha' + \beta' \leq \pi/3$. For each $i$, we have exactly two horoballs in $\eta$ connecting $B_i^L$ to $B_i^U$ with centers in the complement of $V$. Let $\delta_i, \varphi_i$ be the visual angles from center$(C_i)$ of these horoballs. Then, cutting out $V$, we have that the sum of the angles satisfies $$2 \pi \leq (\beta + \alpha) + \delta_1 + \varphi_1 + (\beta' + \alpha') + \delta_2 + \varphi_2 \leq  \frac{\pi}{3}+ \frac{\pi}{3}+ \frac{\pi}{3}+ \frac{\pi}{3}+ \frac{\pi}{3}+ \frac{\pi}{3} = 2\pi.$$
It follows that $\delta_i = \varphi_i = \pi/3$ and $\alpha+\beta = \alpha' + \beta' = \pi/3$. Thus, all horoballs in $\eta$ are full-sized and tangent to $C_1$ or $C_2$. Hence, all the horoballs in $\eta$ are tangent to $C_i$ are part of the hexagonal packing. This allows us to compute $\alpha = \pi - \delta_1 - \varphi_1 - \beta' = \pi/3 - \beta'$ and, similarly, $\beta = \pi/3 - \alpha'$. Since $\alpha + \beta = \pi/3$ and $\alpha' + \beta' = \pi/3$, we obtain a one-parameter family of horoballs parametrized by, say, $\alpha$. 

Without loss of generality, the remaining case is where $B_1^U$ is tangent to $V_1$ and $B_2^U$ is tangent to $V_2$
(with the $x$ coordinate max/min condition above). There is then at least one horoball from $B_1^U$ to $B_2^U$ in the clockwise direction along $\eta$. Let $D_{1,1}, D_{1,2}$ be the next two horoballs in the counter-clockwise direction from $B_1^U$ and $D_{2,1}, D_{2,2}$ the next two horoballs in the clockwise direction from $B_2^U$ along $\eta$. If the visual angle of at least one of $D_{i,j}$ from center$(C_i)$ is $< \pi/3$, then $D_{i,2} \neq B_i^L$. Counting the horoballs tells us that at least one of $B_i^L$ has to be tangent to $V_i$. In fact, both must. Indeed, if $B_2^L$ is tangent to $V_2$ then it cannot be tangent to $B_1^L$ and one more horoball is required in the clockwise direction. Similarly, if $B_1^L$ is tangent to $V_1$. Thus, $D_{i,2} = B_i^L$ for $i  = 1,2$ and $D_{i,j}$ have visual angle $\pi/3$, which means they are full-sized and tangent to $C_i$. The horoballs that connect $B_1^U$ to $B_2^U$ and $B_2^U$ to $B_1^U$ in the clockwise direction must also be full-sized and tangent to both $C_1$ and $C_2$ to bridge the ``width'' of $V$. Thus, we are in the configuration above where $\alpha = \pi/3$.

\bibliographystyle{amsalpha}
\bibliography{biblio.bib}

\end{document}